\documentclass[conference, 10pt, twocolumn]{ieeeconf}       
\bstctlcite{bstctl:etal, bstctl:nodash, bstctl:simpurl}
\def\BibTeX{{\rm B\kern-.05em{\sc i\kern-.025em b}\kern-.08em
    T\kern-.1667em\lower.7ex\hbox{E}\kern-.125emX}}
    
\usepackage[left=54pt, right=54pt,bottom=54pt, top=54pt]{geometry}

\usepackage{amsmath,mathrsfs,amsfonts,amssymb,graphicx,epsfig,booktabs}
 \usepackage{amsthm}
\usepackage{subcaption}
\usepackage{color,multirow,rotating,tcolorbox}
\usepackage{algorithm,algpseudocode,algorithmicx}
\usepackage{cite,url,framed,bm,balance,nicematrix,physics,stfloats}
\usepackage{stmaryrd,mathtools}

\newtheorem{theorem}{Theorem}
\newtheorem{corollary}[theorem]{Corollary}
\newtheorem{definition}{Definition}

\newtheorem{remark}{Remark}

\usepackage{array}
\newtheorem{mytheorem}{Theorem}[]      

\usepackage{tikz}
\usetikzlibrary{calc,trees,positioning,arrows,chains,shapes.geometric,decorations.pathreplacing,decorations.pathmorphing,shapes, matrix,shapes.symbols}

\usepackage[breaklinks]{hyperref}

\renewcommand{\det}{{\mathrm{det}}}

\makeatletter
\newcommand{\pushright}[1]{\ifmeasuring@#1\else\omit\hfill$\displaystyle#1$\fi\ignorespaces}

\newcommand{\dotminus}{\mathbin{\text{\@dotminus}}}

\newcommand{\@dotminus}{%
  \ooalign{\hidewidth\raise1ex\hbox{.}\hidewidth\cr$\m@th-$\cr}%
}

\allowdisplaybreaks

\title{\LARGE\textbf{Symmetrizing Bregman Divergence on the Cone of Positive Definite Matrices: Which Mean to Use and Why
}}

\author{Tushar Sial, Abhishek Halder
\thanks{Tushar Sial and Abhishek Halder are with the Department of Aerospace Engineering, Iowa State University, Ames, IA 50011, USA, {\tt\footnotesize \{tsial,ahalder\}@iastate.edu}.}
\thanks{This research was partially supported by NSF award 2111688.}
}

\IEEEoverridecommandlockouts
\begin{document}
\bstctlcite{IEEEexample:BSTcontrol}
\maketitle
\thispagestyle{empty}
\pagestyle{empty}

\begin{abstract}
This work uncovers variational principles behind symmetrizing the Bregman divergences induced by generic mirror maps over the cone of positive definite matrices. We show that computing the canonical means for this symmetrization can be posed as minimizing the desired symmetrized divergences over a set of mean functionals defined axiomatically to satisfy certain properties. For the forward symmetrization, we prove that the arithmetic mean over the primal space is canonical for any mirror map over the positive definite cone. For the reverse symmetrization, we show that the canonical mean is the arithmetic mean over the dual space, pulled back to the primal space. Applying this result to three common mirror maps used in practice, we show that the canonical means for reverse symmetrization, in those cases, turn out to be the arithmetic, log-Euclidean and harmonic means. Our results improve understanding of existing symmetrization practices in the literature, and can be seen as a navigational chart to help decide which mean to use when.
\end{abstract}

\section{Introduction}\label{sec:Intro}
The Bregman divergence \cite{bregman1967relaxation}, \cite[Ch. 2.1]{censor1997parallel}, \cite{banerjee2005clustering,lewis1996convex,dhillon2008matrix} and its symmetrized variants \cite{nielsen2019jensen,endres2003new} have found widespread usage across control, optimization, and machine learning \cite{bauschke1997legendre,nock2012mining,wang2014bregman,harandi2014bregman,halder2019degroot,halder2020hopfield,hassibi2025beyond}. Intuitively, such divergences are squared distance-like functionals \cite[p. 10]{amari2016information} over appropriate domains. These functionals are not symmetric in general, but can be made so by a symmetrization procedure. This work focuses on the choice of mean used for the symmetrization of a class of Bregman divergence.

To distill the problem, let us first introduce the background ideas and notations.

\noindent\textbf{Geometry of Positive Definite Matrices.}\\
\noindent Let $\mathbb{S}^{n}_{++}$ denote the open convex cone of $n\times n$ positive definite matrices, i.e.,
\begin{align*}
 \mathbb{S}^{n}_{++} :=\{X\in\mathbb{R}^{n\times n} \mid X^{\top}=X, &\; v^{\top}X v > 0\\
 &\forall \:\text{nonzero vector}\;v\in\mathbb{R}^{n}\}.   
\end{align*}
Likewise, $\mathbb{S}^{n}_{--}$ denotes the convex cone of $n\times n$ negative definite matrices.

Any $X\in\mathbb{S}^{n}_{++}$ admits polar decomposition \cite[Thm. 7.3.1]{horn2012matrix}
\begin{align}
X = YY^{\top}=(UZ)(UZ)^{\top}=U Z^2 U^{\top},
\label{PolarDecomposition}
\end{align}
where $Y\in{\mathrm{GL}}(n)$ (the set of $n\times n$ invertible matrices with real entries), $U\in{\mathrm{O}}(n)$ (the set of $n\times n$ orthogonal matrices), and $Z\in\mathbb{S}^{n}_{++}$. The decomposition \eqref{PolarDecomposition} highlights the quotient geometry
$$\mathbb{S}^{n}_{++}\cong {\mathrm{GL}}(n)/{\mathrm{O}}(n)\cong\left({\mathrm{O}}(n)\times\mathbb{S}^{n}_{++}\right)/{\mathrm{O}}(n),$$
which implies a ${\mathrm{GL}}(n)$-invariant Riemannian metric on $\mathbb{S}^{n}_{++}$, referred to as the \emph{natural metric} \cite{faraut1994analysis,bonnabel2010riemannian}. Recall that at any $A\in\mathbb{S}^{n}_{++}$, the tangent space ${\mathrm{T}}_{A}\mathbb{S}^{n}_{++}\subseteq\mathbb{S}^{n}$ (the set of $n\times n$ symmetric matrices), and the ${\mathrm{GL}}(n)$-invariant metric at the tangent space at identity is equipped with the Frobenius inner product, i.e., $g_{I}(X,Y):=\langle X,Y\rangle = \tr(XY)$. So the ${\mathrm{GL}}(n)$-invariant metric at arbitrary $A\in\mathbb{S}^{n}_{++}$ must be given by \cite{moakher2005differential,bonnabel2010riemannian,halder2016finite}
\begin{align}
g_{A}(X,Y):=\langle X,\left(\mathcal{L}_{A}\circ\mathcal{R}_{A}\right)^{-1}Y\rangle ={\mathrm{trace}}\!\left(\!A^{-1}XA^{-1}Y\right),
\label{RiemannianMetric}    
\end{align}
where $\mathcal{L}_{A}X:=AX$ (left multiplication by $A$), and $\mathcal{R}_{A}X:=XA$ (right multiplication by $A$).

Endowed with \eqref{RiemannianMetric}, we view $\mathbb{S}^{n}_{++}$ as a Riemannian manifold of nonpositive curvature \cite[p. 226-227]{bhatia2009positive}. In this manifold, any $\mathcal{C}^{1}$ curve $\gamma:[0,1]\to\mathbb{S}^{n}_{++}$ has length $\int_{0}^{1}\|\gamma^{-\frac{1}{2}}(t)\gamma^{\prime}(t)\gamma^{-\frac{1}{2}}(t)\|_{\mathrm{F}}\:\differential t$ where $\|\cdot\|_{\mathrm{F}}$ denotes the Frobenius norm, and $^{\prime}$ denotes the derivative. For fixed $A,B\in\mathbb{S}^{n}_{++}$, the minimal Riemannian geodesic curve
\begin{align}
\gamma^{\mathrm{opt}}(t) := &\underset{\gamma(\cdot)\in\mathcal{C}^{1}\left([0,1];\:\mathbb{S}^{n}_{++}\right)}{\arg\inf}\displaystyle\int_{0}^{1} \|\gamma^{-\frac{1}{2}}(t)\gamma^{\prime}(t)\gamma^{-\frac{1}{2}}(t)\|_{\mathrm{F}}\:\differential t,\nonumber\\
&\quad\mathrm{subject\,to}\quad \gamma(0)=A, \quad \gamma(1)=B.
\label{defRiemannianGeodesic}    
\end{align}
In particular \cite{pusz1975functional,ando2004geometric}, 
\begin{align}
\gamma^{\mathrm{opt}}(t) = A^{\frac{1}{2}}\left(A^{-\frac{1}{2}}BA^{-\frac{1}{2}}\right)^{t}A^{\frac{1}{2}} \quad \forall t\in[0,1].
\label{MinimalRiemannianGeodesic}    
\end{align}
\noindent\textbf{Bregman divergence on $\mathbb{S}^{n}_{++}$.}\\
\noindent The Bregman divergence on $\mathbb{S}^{n}_{++}$ is induced by a mapping $\psi:\mathbb{S}_{++}^{n}\mapsto\mathbb{R}$, where $\psi$ is a strictly convex $\mathcal{C}^{2}$ function of Legendre type \cite[p. 258]{tyrrell1970convex}. In the mirror descent literature \cite[Sec. 4.1]{bubeck2015convex}, \cite{nemirovskij1983problem,halder2019degroot}, $\psi$ is referred to as the \emph{mirror map}. The formal definition for Bregman divergence is as follows. 
\begin{definition}[Bregman divergence on $\mathbb{S}^{n}_{++}$]\label{def:BregmanDivergence}
Let $\psi$ be a closed, proper, strictly convex, and twice continuously differentiable function on $\mathbb{S}^{n}_{++}$, $\nabla\psi(X)$ is a diffeomorphism $\forall X\in\mathbb{S}^{n}_{++}$, and $\|\nabla\psi(X)\|_2 \rightarrow\infty$ as $X\rightarrow{\mathrm{boundary}}({\mathrm{closure}}(\mathbb{S}^{n}_{++}))$. The \emph{Bregman divergence induced by $\psi$} is a mapping $D_{\psi}:\mathbb{S}_{++}^{n}\times\mathbb{S}_{++}^{n}\mapsto\mathbb{R}_{\geq 0}$ given by
\begin{align}
D_{\psi}(X,Y) := \psi(X) - \bigg\{\psi(Y)+ \bigg\langle\frac{\partial\psi}{\partial Y},X-Y\bigg\rangle\bigg\},
\label{defBregmanDivOnPosDefCone}    
\end{align}
where $\langle\cdot,\cdot\rangle$ denotes the Frobenius inner product. 
\end{definition}

\begin{remark}\label{remark:LegendreType}
That $\psi$ is of Legendre-type on $\mathbb{S}^{n}_{++}$, as standard in mirror map definition, is stronger than requiring closed, proper, strictly convex and $\mathcal{C}^2$. This is because Legendre-type $\mathcal{C}^2$ additionally requires $\nabla\psi$ to be a diffeomorphism (differentiable bijection with differentiable inverse) and its magnitude to be unbounded at the boundary of the closure of $\mathbb{S}^{n}_{++}$. For instance, if $\psi$ were only closed, proper, strictly convex and $\mathcal{C}^2$, then we can guarantee its Hessian ${\mathrm{H}}\psi$ to be positive semidefinite everywhere in its domain. The additional requirement of $\nabla\psi$ to be a diffeomorphism implies (by inverse function theorem) that ${\mathrm{H}}\psi$ is, in fact, positive definite on the open convex set $\mathbb{S}^{n}_{++}$.
\end{remark}

 Geometrically, $D_{\psi}$ quantifies the error at $X\in\mathbb{S}^{n}_{++}$ due to first-order Taylor approximation of $\psi$ at $Y\in\mathbb{S}^{n}_{++}$. 
 
 It is straightforward to see that $D_{\psi}\geq 0$ $\forall X,Y\in\mathbb{S}^{n}_{++}$, and $D_{\psi}=0$ iff $X=Y\in\mathbb{S}^{n}_{++}$. Yet $D_{\psi}$ is not a metric in general\footnote{An exception is the case $\psi(X) = \|X\|_{\mathrm{F}}^2$ (squared Frobenius norm) which yields $D_{\psi}(X,Y) = \|X-Y\|_{\mathrm{F}}^2$, see first row of Table \ref{table:BregmanExamples}.} due to lack of symmetry (i.e., $D_{\psi}(X,Y)\neq D_{\psi}(Y,X)$) and failure to satisfy triangle inequality. Interestingly, $D_{\psi}$ is a convex function w.r.t. its first argument.

\begin{table}[!t]
\centering
\caption{Mirror maps and Bregman divergences on $\mathbb{S}^{n}_{++}$.}
\label{table:BregmanExamples}
\begin{tabular}{ c | c }
\toprule 
Mirror map $\psi$ & Bregman divergence $D_{\psi}$\\[0.4ex]
\hline
$\|X\|_{\mathrm{F}}^{2}$ & $\|X-Y\|_{\mathrm{F}}^{2}$\\[0.1ex]
${\mathrm{trace}}(X\log X - X)$ & ${\mathrm{trace}}(X\log X -X\log Y - X + Y)$\\[0.1ex]
$-\log\det X$ & ${\mathrm{trace}}(XY^{-1}) - \log\det(XY^{-1}) - n$\\
\bottomrule 
\end{tabular}
\end{table} 

Table \ref{table:BregmanExamples} lists well-known \cite{lewis1996convex,dhillon2008matrix} examples of $\psi, D_{\psi}$ on $\mathbb{S}^{n}_{++}$. See also \cite[Table 15.1]{nock2012mining}. In particular, the mirror map $\psi$ in the \emph{first row} of Table \ref{table:BregmanExamples} is the squared Frobenius norm, which induces the squared Frobenius norm of the difference as $D_{\psi}$. 

The mirror map $\psi$ in the \emph{second row} of Table \ref{table:BregmanExamples} is the negative von Neumann entropy \cite[Ch. 13]{bengtsson2017geometry}, which induces Umegaki's quantum relative entropy \cite{umegaki1962conditional,hiai1991proper} as $D_{\psi}$.

The mirror map $\psi$ in the \emph{third row} of Table \ref{table:BregmanExamples} is the Burg's entropy\footnote{Burg's entropy is the natural self-concordant barrier for $\mathbb{S}^{n}_{++}$, and is used in interior point algorithms \cite[Ch. 5.4]{nesterov1994interior}.} \cite{burg1975maximum}, which induces (twice of) the classical relative entropy or the Kullback-Leibler divergence between centered Gaussian distributions as $D_{\psi}$.

\noindent\textbf{Symmetrization of $D_{\psi}$ on $\mathbb{S}^{n}_{++}$.}\\
\noindent One way to fix the lack of symmetry in \eqref{defBregmanDivOnPosDefCone} is to consider the \emph{Jeffrey's symmetrization} \cite{nielsen2019jensen}: $\frac{1}{2}\{D_{\psi}(X\parallel Y) + D_{\psi}(Y\parallel X)\}$. In this work, we focus on a different class of \emph{symmetrized Bregman divergence}, defined as
\begin{align}
&D_{\psi}^{\overrightarrow{\mathrm{symm}}}(X,Y;M)\nonumber\\
&:= \dfrac{1}{2}\bigg\{D_{\psi}\left(X\parallel M(X,Y)\right) + D_{\psi}\left(Y\parallel M(X,Y)\right)\bigg\},
\label{defSymmetrizedBregmanDiv}
\end{align}
or as
\begin{align}
&D_{\psi}^{\overleftarrow{\mathrm{symm}}}(X,Y;M)\nonumber\\
&:= \dfrac{1}{2}\bigg\{D_{\psi}\left(M(X,Y)\parallel X\right) + D_{\psi}\left(M(X,Y)\parallel Y\right)\bigg\},
\label{defReverseSymmetrizedBregmanDiv}
\end{align}
where $M(\cdot,\cdot)$ is a \emph{mean} on $\mathbb{S}^{n}_{++}$ given axiomatically in Definition \ref{defGeneralizedMean} below. Notice that the ``forward'' symmetrization \eqref{defSymmetrizedBregmanDiv}, and  the ``reverse'' symmetrization \eqref{defReverseSymmetrizedBregmanDiv} yield different functionals in general, due to the lack of symmetry of $D_{\psi}$.

\begin{definition}[Mean on $\mathbb{S}^{n}_{++}$]\label{defGeneralizedMean}\cite[Ch. 4, p. 101]{bhatia2009positive}
A mapping $M:\mathbb{S}^{n}_{++}\times \mathbb{S}^{n}_{++} \mapsto \mathbb{S}^{n}_{++}$ is called a mean if:
\begin{itemize}
\item[(i)] $M(A,B) \succ 0$,
\item[(ii)] $M(A,B) = M(B,A)$,
\item[(iii)] $A \preceq B \,\Rightarrow\; A \preceq M(A,B) \preceq B$,
\item[(iv)] $M(A,B)$ is operator monotone increasing in $A,B$,
\item[(v)] $M(P^{\top}A P, P^{\top}BP) = P^{\top}M(A,B)P\quad\forall P\in{\rm{GL}}(n)$,
\item[(vi)] $M(A,B)$ is continuous in $A,B$,
\end{itemize}
for all $A,B\in\mathbb{S}^{n}_{++}$, where the curved inequalities are understood in the L\"{o}wner partial order sense.
\end{definition}
Definition \ref{defGeneralizedMean} is closely related to a finite-dimensional symmetric
Kubo–Ando operator mean \cite{kubo1980means}. In the special case $n=1$, $\mathbb{S}^{n}_{++}\equiv\mathbb{R}_{>0}$. In that case, using lowercase for scalars, familiar examples of $M$ are arithmetic mean, geometric mean, harmonic mean, logarithmic mean, given respectively as 
$$\frac{a+b}{2},\quad\sqrt{ab},\quad\left(\frac{a^{-1} + b^{-1}}{2}\right)^{\!\!-1}\!\!\!\!,\quad\int_{0}^{1}a^{t}b^{1-t}\differential t.$$
Per Definition \ref{defGeneralizedMean}, these examples generalize on $\mathbb{S}^{n}_{++}$, respectively, as 
\begin{align*}
\dfrac{A+B}{2},\,A^{\frac{1}{2}}\!\!\left(\!A^{-\frac{1}{2}}BA^{-\frac{1}{2}}\!\right)^{\!\!\frac{1}{2}}\!\!A^{\frac{1}{2}}, \, \left(\!\!\dfrac{A^{-1} + B^{-1}}{2}\!\right)^{\!\!\!-1}\!\!\!\!, \int_{0}^{1}\!\!\!A^{t}B^{1-t}\differential t. 
\end{align*}

\begin{remark}[Geometric mean]\label{Reamrk:GeometricMean}
The geometric mean $A^{\frac{1}{2}}\!\!\left(\!\!A^{-\frac{1}{2}}BA^{-\frac{1}{2}}\!\!\right)^{\!\frac{1}{2}}\!\!A^{\frac{1}{2}}$ reduces to $A^{\frac{1}{2}}B^{\frac{1}{2}}$ iff $A,B$ commute\footnote{In general, $A^{\frac{1}{2}}B^{\frac{1}{2}}$ is not symmetric for $A,B\in\mathbb{S}^{n}_{++}$.}. From \eqref{MinimalRiemannianGeodesic}, this non-commutative generalization of geometric mean is precisely the midpoint of the minimal Riemannian geodesic in $\mathbb{S}^{n}_{++}$, i.e., $\gamma^{\mathrm{opt}}(\frac{1}{2})$. A commonly used notation is
\begin{align}
A\:\#_{t}\:B := \gamma^{\mathrm{opt}}\left(t\right)\:\forall t\in[0,1], \,\text{and}\,A\:\#\:B := \gamma^{\mathrm{opt}}\left(\frac{1}{2}\right).
\label{HashNotation}    
\end{align}
\end{remark}

\begin{remark}[Logarithmic mean]\label{Reamrk:LogarithmicMean}
In the scalar case, the term logarithmic mean refers to the fact that for $a,b>0$, the integral $\int_{0}^{1}a^{t}b^{1-t}\differential t$ equals $ \frac{a-b}{\log a - \log b}$ for $a\neq b$, and equals $a$ for $a=b$. In the $n>1$ case, the integral does not simplify for general non-commutative $A,B\in\mathbb{S}^{n}_{++}$. 
\end{remark}

\begin{remark}[Log-Euclidean mean]\label{Reamrk:LogEuclideanMean}
Different from the geometric mean $A\:\#\:B$, is the log-Euclidean mean
\begin{align}
\exp\left(\dfrac{\log A + \log B}{2}\right),
\label{LogEuclideanMean}    
\end{align}
where $\log$ denotes the principal logarithm. This mean satisfies all properties in Definition \ref{defGeneralizedMean} except that (v) holds only for $P\in{\mathrm{O}}(n)$. The geometric and log-Euclidean means become equal (to the value $A^{\frac{1}{2}}B^{\frac{1}{2}}$) iff $A,B$ commute.
\end{remark}

Motivated by Remark \ref{Reamrk:LogEuclideanMean}, we define the following two sets of means.
\begin{definition}[${\mathrm{GL}}(n)$ and ${\mathrm{O}}(n)$ invariant means]\label{def:InvariantMeans}
Let $\mathcal{M}$ denote the set of means on $\mathbb{S}^{n}_{++}$ satisfying Definition \ref{defGeneralizedMean}. Elements of $\mathcal{M}$ are called ${\mathrm{GL}}(n)$ invariant due to property (v) in Definition \ref{defGeneralizedMean}. Let $\widehat{\mathcal{M}}$ denote the set of means on $\mathbb{S}^{n}_{++}$ satisfying all of Definition \ref{defGeneralizedMean} except that the ``$\forall P\in{\mathrm{GL}}(n)$'' in (v) replaced by ``$\forall P\in{\mathrm{O}}(n)$''. Elements of $\widehat{\mathcal{M}}$ are called ${\mathrm{O}}(n)$ invariant. We note the set inclusion:
\begin{align}
\mathcal{M}\subset\widehat{\mathcal{M}}.
\label{SetInclusion}    
\end{align}
\end{definition}
\noindent To exemplify Definition \ref{def:InvariantMeans}, we note that the arithmetic, geometric, harmonic and logarithmic means over $\mathbb{S}^{n}_{++}$ are in $\mathcal{M}$. The log-Euclidean mean \eqref{LogEuclideanMean} is in $\widehat{\mathcal{M}}$ but not in $\mathcal{M}$. 

\noindent\textbf{Motivation.}\\
\noindent Now that we have the definition and examples for the generalized mean $M$ in place, let us return to \eqref{defSymmetrizedBregmanDiv}. We think of $D_{\psi}^{\overrightarrow{\mathrm{symm}}}$ in \eqref{defSymmetrizedBregmanDiv} to be parameterized by a choice of generalized mean $M$. If we choose $M$ to be the \emph{arithmetic mean} in \eqref{defSymmetrizedBregmanDiv}, then $D_{\psi}^{\overrightarrow{\mathrm{symm}}}$ reduces to the well-known \emph{Stein or Jensen-Shannon symmetrization} \cite{nielsen2019jensen}
\begin{align}
\dfrac{\psi(X) + \psi(Y)}{2} -\psi\left(\dfrac{X+Y}{2}\right),
\label{JensenShannonSymmetrization}
\end{align}
which follows from \eqref{defBregmanDivOnPosDefCone}. In particular, \cite{sra2016positive} showed that specializing \eqref{JensenShannonSymmetrization} with $\psi(X) = -\log\det X$ yields  a squared distance functional named therein as the $S$-divergence\footnote{The $S$-divergence can be interpreted as the Bhattacharya distance \cite{bhattacharyya1943measure} between two centered Gaussian probability distributions, but is not a metric. See also \cite{cherian2011efficient}.}. In other words, the square root of \eqref{JensenShannonSymmetrization} then is a metric on $\mathbb{S}^{n}_{++}$.

These observations raise natural questions:
\begin{itemize}
\item is there something special about the choice of $M$ as arithmetic mean in \eqref{defSymmetrizedBregmanDiv} as the literature on specific cases seem to suggest?

\item what should be the canonical choice of $M$ in \eqref{defReverseSymmetrizedBregmanDiv}?

\item should the choice of canonical $M$ be affected by the specificity of $\psi$?
\end{itemize} 
The mathematical developments in this work are motivated by these questions.

\noindent\textbf{Contributions and related works.}\\
\noindent The purpose of this work is to clarify canonical choices of $M$ for \eqref{defSymmetrizedBregmanDiv} and \eqref{defReverseSymmetrizedBregmanDiv}. Specifically, we formulate variational principles that guide the canonical choices of the mean $M$ independent of $\psi$.

To the best of our knowledge, the questions we posed above regarding the symmetrization \eqref{defSymmetrizedBregmanDiv} and \eqref{defReverseSymmetrizedBregmanDiv} over $\mathbb{S}^{n}_{++}$, have not been investigated before. The related work in \cite{nielsen2019jensen} is in a different spirit: it focuses on a statistical viewpoint and considers a specific symmetrization and specific $\psi$ resulting in the  Kullback-Leibler divergence.

By answering the aforesaid questions, we improve the understanding of Bregman symmetrization, specifically by showing that the choice of $M$ is not ad hoc, but is principled by variational reasoning. 

\noindent\textbf{Organization.}\\
\noindent The remaining of this work is structured as follows. Sec. \ref{sec:MainResults} derives the main results, viz. the optimal choice of canonical means (Theorem \ref{thm:arthmeticmean} and Theorem \ref{thm:geometricmean}), and the corresponding canonical forward and reverse symmetrized Bregman divergences (Corollary \ref{corollary:Forward} and Corollary \ref{corollary:Backward}). Specifically, Theorem \ref{thm:arthmeticmean} shows that for forward symmetrization, the arithmetic mean is canonical irrespective of the choice of mirror map $\psi$. Theorem \ref{thm:geometricmean} establishes that for reverse symmetrization, the canonical mean is obtained by first computing the arithmetic mean in dual space, and then pulling it back to the primal space. Sec. \ref{sec:examples} details the computational specifics for the three examples in Table \ref{table:BregmanExamples}, thereby illustrating the results from Sec. \ref{sec:MainResults} for those three commonly used mirror maps on $\mathbb{S}^{n}_{++}$. Sec. \ref{sec:conclusion} concludes the work.


\section{Main Results}\label{sec:MainResults}
We seek to choose the ``best'' $M$ for fixed $X,Y\in\mathbb{S}^{n}_{++}$ for the symmetrized Bregman divergence functionals \eqref{defSymmetrizedBregmanDiv} or \eqref{defReverseSymmetrizedBregmanDiv}. Since these functionals are, by definition, non-negative and symmetric, it is natural to define the canonical mean $M$ as the least conservative choice, i.e., the $M$ which \emph{minimize} \eqref{defSymmetrizedBregmanDiv} or \eqref{defReverseSymmetrizedBregmanDiv}. Thus motivated, we define\footnote{Problems \eqref{McanForward}-\eqref{McanBackward} are, in the general (non-matrix) Bregman setting, instances of \emph{left-} and \emph{right-sided Bregman centroids} studied in \cite{nielsen2009sided}; our contribution is in identifying these centroids with specific matrix means on $\mathbb{S}^{n}_{++}$.}
\begin{subequations}
\begin{align}
&\overrightarrow{M}_{\mathrm{canonical}} := \underset{M\in\mathcal{M}}{\arg\min}\; D_{\psi}^{\overrightarrow{\mathrm{symm}}}(X,Y;M),\label{McanForward}\\
&\overleftarrow{M}_{\mathrm{canonical}} := \underset{M\in\widehat{\mathcal{M}}}{\arg\min}\; D_{\psi}^{\overleftarrow{\mathrm{symm}}}(X,Y;M).\label{McanBackward}
\end{align}
\label{defCanonicalM}
\end{subequations}
The reason for using $\widehat{\mathcal{M}}$ instead of $\mathcal{M}$ in \eqref{McanBackward} will be explained in Sec. \ref{subsec:GeometricMean} (proof of Theorem \ref{thm:geometricmean} and Remark \ref{remark:MversusMhat}). At this point, it is not \emph{a priori} clear if such minimizers exist or unique or even if they do, whether they depend on the choice of $\psi$ in addition to the fixed data $X,Y\in\mathbb{S}^{n}_{++}$. Last but not the least, it is unclear how to perform the minimization since it is not apparent how to parameterize $\mathcal{M}$ or $\widehat{\mathcal{M}}$.


\subsection{\!\!Arithmetic Mean is Canonical for Forward Symmetrization}\label{subsec:ArithmeticMean}
Our first result concerns with \eqref{McanForward}, i.e., the choice of canonical mean in \emph{forward} Bregman symmetrization \eqref{defSymmetrizedBregmanDiv}.
\begin{theorem}\label{thm:arthmeticmean}
Consider the forward Bregman symmetrization \eqref{defSymmetrizedBregmanDiv} with fixed $X,Y\in\mathbb{S}^{n}_{++}$. The unique minimizer in \eqref{McanForward} is
\begin{align}
\overrightarrow{M}_{\mathrm{canonical}} = \dfrac{X+Y}{2}.
\label{ForwardMinimizer}    
\end{align}
\end{theorem}
\begin{proof}
Combining \eqref{defBregmanDivOnPosDefCone} and \eqref{defSymmetrizedBregmanDiv}, we write
\begin{align}
&D_{\psi}^{\overrightarrow{\mathrm{symm}}}(X,Y;M) = \frac{1}{2}\psi(X)+\frac{1}{2}\psi(Y)-\psi(M)\nonumber\\
&\qquad\qquad-\frac{1}{2}\bigg\langle\frac{\partial\psi}{\partial M},X+Y\bigg\rangle+\bigg\langle\frac{\partial\psi}{\partial M},M\bigg\rangle.
\label{ForwardIntermsOfpsi}    
\end{align}
Adding and subtracting $\psi\left(\dfrac{X+Y}{2}\right)$ in the right hand side of \eqref{ForwardIntermsOfpsi}, we obtain
\begin{align}
D_{\psi}^{\overrightarrow{\mathrm{symm}}}(X,Y;M) =  &\dfrac{\psi(X) + \psi(Y)}{2} -\psi\left(\dfrac{X+Y}{2}\right) \nonumber\\
& + D_{\psi}\left(\dfrac{X+Y}{2}\parallel M\right),
\label{SumOfTwoTerms}
\end{align}
i.e., $D_{\psi}^{\overrightarrow{\mathrm{symm}}}$ decomposes as the Jensen gap \eqref{JensenShannonSymmetrization} plus the Bregman
divergence between the arithmetic mean and the candidate mean. From the property of Bregman divergence, \eqref{SumOfTwoTerms} has a unique global minimizer $M=\dfrac{X+Y}{2}$.  
Since this solution is an element of the feasible set $\mathcal{M}$ in \eqref{McanForward}, the claim follows.
\end{proof}
\noindent The next Corollary follows from \eqref{SumOfTwoTerms}.
\begin{corollary}\label{corollary:Forward}
The canonical forward symmetrized Bregman divergence on $\mathbb{S}^{n}_{++}$ is the \emph{Stein or Jensen-Shannon symmetrization} \eqref{JensenShannonSymmetrization}.
\end{corollary}


\subsection{Arithmetic Mean in the Dual Space is Canonical for Reverse Symmetrization}\label{subsec:GeometricMean}
Our next result concerns with \eqref{McanBackward}, i.e., the choice of canonical mean in \emph{reverse} Bregman symmetrization \eqref{defReverseSymmetrizedBregmanDiv}. 

Unlike Theorem \ref{thm:arthmeticmean}, the next result requires further assumptions on $\psi$.

\begin{theorem}\label{thm:geometricmean}
Consider the reverse Bregman symmetrization \eqref{defReverseSymmetrizedBregmanDiv} with fixed $X,Y\in\mathbb{S}^{n}_{++}$, and with the following additional assumptions on the mirror map:
\begin{itemize}
\item $\nabla\psi$ is operator monotone\footnote{A matrix function $F(X)$ is called operator monotone if $X\preceq Y\:\Rightarrow\:F(X)\preceq F(Y)$.},
\item $\psi$ is spectral a.k.a. isotropic, i.e., $\psi(P^{\top}XP) = \psi(X)$ $\forall P\in{\mathrm{O}}(n)$.
\end{itemize}
The unique minimizer in \eqref{McanBackward} is
\begin{align}
\overleftarrow{M}_{\mathrm{canonical}} = \nabla\psi^{*}\left(\dfrac{\nabla\psi(X) + \nabla\psi(Y)}{2}\right),
\label{BackwardMinimizer}    
\end{align}
where $\psi^{*}(Y):=\underset{X\in\mathbb{S}^{n}_{++}}{\sup}\big\{\langle Y,X\rangle - \psi(X)\big\}$ is the Legendre-Fenchel conjugate of $\psi$.
\end{theorem}
\begin{proof}

Combining \eqref{defBregmanDivOnPosDefCone} and \eqref{defReverseSymmetrizedBregmanDiv}, we write
\begin{align}
&D_{\psi}^{\overleftarrow{\mathrm{symm}}}(X,Y;M) = \psi(M) -\frac{1}{2}\bigg\langle\frac{\partial\psi}{\partial X}+\frac{\partial\psi}{\partial Y},M\bigg\rangle\nonumber\\
&-\frac{1}{2}\psi(X)-\frac{1}{2}\psi(Y)+\frac{1}{2}\bigg\langle\frac{\partial\psi}{\partial X},X\bigg\rangle + \frac{1}{2}\bigg\langle\frac{\partial\psi}{\partial Y},Y\bigg\rangle.
\label{ReverseIntermsOfpsi}    
\end{align}
Thus,
\begin{align}
\dfrac{\partial}{\partial M}D_{\psi}^{\overleftarrow{\mathrm{symm}}} = \dfrac{\partial\psi}{\partial M} - \frac{1}{2}\left(\frac{\partial\psi}{\partial X}+\frac{\partial\psi}{\partial Y}\right).
\label{FirstDerivativeOfReverse}
\end{align}

From \eqref{FirstDerivativeOfReverse}, the necessary condition for optimality yields
\begin{align}
\overleftarrow{M}_{\mathrm{canonical}} = \left(\nabla\psi\right)^{-1}\left(\dfrac{\nabla\psi(X) + \nabla\psi(Y)}{2} \right),
\label{McanReverseSemi}    
\end{align}
which is well-defined and unique because $\psi$ being of Legendre-type, $\nabla\psi$ is a bijection from $\mathbb{S}^{n}_{++}$ to its dual space (range of the gradient): a subset of $\mathbb{S}^{n}$. Recalling the Fenchel duality result \cite[Sec. 2.2]{nielsen2007bregman} $(\nabla\psi)^{-1}=\nabla\psi^{*}$, the expression \eqref{BackwardMinimizer} follows from \eqref{McanReverseSemi}. Furthermore, the unique critical point \eqref{McanReverseSemi} (equivalently  \eqref{BackwardMinimizer}) is a minimizer because \eqref{ReverseIntermsOfpsi} is strictly convex in candidate mean.

We now verify that \eqref{BackwardMinimizer} (equivalently \eqref{McanReverseSemi}) belongs to $\widehat{\mathcal{M}}$. It is obvious that \eqref{McanReverseSemi} satisfies properties (i) and (ii) in Definition \ref{defGeneralizedMean}. Property (vi) therein is also satisfied since \eqref{McanReverseSemi} is a composition of continuous (in fact $\mathcal{C}^{1}$) maps, and $\psi$ is of Legendre-type. The properties (iii)-(iv) hold only if $\nabla\psi$ is operator monotone, which is an additional assumption\footnote{For the three mirror maps $\psi(X)$ listed in Table \ref{table:BregmanExamples}, the corresponding $\nabla\psi(X)$ are $2X$, $\log X$ (principal logarithm), $-X^{-1}$, respectively, and are all operator monotone. See e.g., \cite[Thm. 2.6]{carlen2010trace}.} on $\psi$ (beyond closed, proper, strictly convex, $\mathcal{C}^2$, Legendre-type) in our statement. That \eqref{McanReverseSemi} belongs to $\widehat{\mathcal{M}}$, and not in $\mathcal{M}$ in general, follows from $\psi$ being spectral\footnote{All three mirror maps $\psi(X)$ listed in Table \ref{table:BregmanExamples} are spectral.}, since then
\begin{align*}
&\left(\nabla\psi\right)^{-1}\left(\dfrac{\nabla\psi(P^{\top}XP) + \nabla\psi(P^{\top}YP)}{2} \right)\\
=&P^{\top}\left(\nabla\psi\right)^{-1}\left(\dfrac{\nabla\psi(X) + \nabla\psi(Y)}{2} \right)P \quad \forall P\in{\mathrm{O}}(n). 
\end{align*}
This completes the proof.
\end{proof}

\begin{remark}\label{remark:intepretation}
Notice that \eqref{ForwardMinimizer} is the \emph{primal} arithmetic mean, i.e., the arithmetic mean in $\mathbb{S}^{n}_{++}$. In contrast,  \eqref{BackwardMinimizer} (equivalently \eqref{McanReverseSemi}) computes the \emph{dual} arithmetic mean, i.e., the arithmetic mean in the dual space $\left(\mathbb{S}^{n}_{++}\right)^{*}=\mathbb{S}^{n}$, which is then pulled back to $\mathbb{S}^{n}_{++}$ via $(\nabla\psi)^{-1}=\nabla\psi^{*}$.
\end{remark}

\begin{remark}\label{remark:MversusMhat}
Even though for generic $\psi$ as in Theorem \ref{thm:geometricmean} statement, we can only guarantee optimality in $\widehat{\mathcal{M}}$, we will see in Sec. \ref{sec:examples} that for two out of the three examples in Table \ref{table:BregmanExamples}, the minimizer \eqref{BackwardMinimizer} will turn out to be from $\mathcal{M}$. In other words, computing the minimizer over $\widehat{\mathcal{M}}$ as in \eqref{McanBackward}, is often not conservative in practice.
\end{remark}

\noindent For sanity check, notice that when $X=Y$, the minimizer \eqref{BackwardMinimizer} gives $\overleftarrow{M}_{\mathrm{canonical}} = (\nabla\psi^{*})(\nabla \psi(X)) = X$, thanks to the Fenchel duality.
\begin{corollary}\label{corollary:Backward}
When the conditions in Theorem \ref{thm:geometricmean} statement hold, the canonical reverse symmetrized Bregman divergence on $\mathbb{S}^{n}_{++}$ is
\begin{align}
&D_{\psi}^{\overleftarrow{\mathrm{symm}}}(X,Y;\overleftarrow{M}_{\mathrm{canonical}})\nonumber\\
&= \underbrace{\frac{\psi^*(\nabla\psi(X))+\psi^*(\nabla\psi(Y))}{2} - \psi^*\!\left(\!\frac{\nabla\psi(X)+\nabla\psi(Y)}{2}\!\right)}_{{\emph{Jensen gap of }}\psi^*\emph{ at }\nabla\psi(X),\nabla\psi(Y)}\!.
\label{CanonicalReverseBregman}    
\end{align}
\end{corollary}
\begin{proof}
Substituting $\mu := \dfrac{\nabla\psi(X)+\nabla\psi(Y)}{2}$ and $M = \overleftarrow{M}_{\mathrm{canonical}}$ in \eqref{ReverseIntermsOfpsi}, we get
{\small{\begin{align}
&D_\psi^{\overleftarrow{\mathrm{symm}}}(X,Y;\overleftarrow{M}_{\mathrm{canonical}})
= \psi(\overleftarrow{M}_{\mathrm{canonical}}) - \langle \mu, \overleftarrow{M}_{\mathrm{canonical}}\rangle
 \nonumber\\
&- \frac{\psi(X)+\psi(Y)}{2}+ \dfrac{\langle \nabla\psi(X),X\rangle + \langle \nabla\psi(Y),Y\rangle}{2}.
\label{eq:step1}
\end{align}
}}
In \eqref{eq:step1}, we now substitute the following two applications of Fenchel duality:
\begin{align*}
&\psi(\overleftarrow{M}_{\mathrm{canonical}}) - \langle \mu, \overleftarrow{M}_{\mathrm{canonical}} \rangle = -\:\psi^*(\mu),\\
&\langle \nabla\psi(Z), Z\rangle = \psi(Z) + \psi^*(\nabla\psi(Z)) \quad \text{for}\;Z\in\{X,Y\},
\end{align*}
to cancel the $\psi(X)/2+\psi(Y)/2$ terms, yielding \eqref{CanonicalReverseBregman}.
\end{proof}
As expected, \eqref{CanonicalReverseBregman} vanishes at $X=Y$.  


\begin{table*}[!t]
\centering
    \caption{Canonical means and symmetrized Bregman divergences on $\mathbb{S}^{n}_{++}$ for the mirror maps $\psi$ in Table \ref{table:BregmanExamples} (the rows are in the same order). We do not list $\overrightarrow{M}_{\mathrm{canonical}}=\frac{X+Y}{2}$ since it holds unconditionally of the choice of $\psi$.}
\label{tableFinalExamples}
\small
{\setlength{\tabcolsep}{6pt}
{
\begin{tabular}{c|c|c|c}
\toprule 
Sec. & $\overleftarrow{M}_{\mathrm{canonical}}$ & $D_{\psi}^{\overrightarrow{\mathrm{symm}}}$ & $D_{\psi}^{\overleftarrow{\mathrm{symm}}}$\\[0.1ex]
\hline
\ref{subsec:SquaredFrob} & $\frac{X+Y}{2}$ & $\frac{1}{4}\|X-Y\|_{\mathrm{F}}^{2}$ & $\frac{1}{4}\|X-Y\|_{\mathrm{F}}^{2}$\\[0.1ex]
\ref{subsec:NegVonNeumannEntropy}  & $\exp\left(\frac{\log X + \log Y}{2}\right)$ & $\frac{1}{2}{\mathrm{trace}}(X\log X + Y\log Y)\!-\!{\mathrm{trace}}\!\left(\!\frac{X+Y}{2}\log\frac{X+Y}{2}\!\right)$  & ${\mathrm{trace}}\left(\dfrac{X+Y}{2}\right)-{\mathrm{trace}}\!\left(\exp\!\left(\frac{\log X+\log Y}{2}\right)\right)$\\[0.1ex]
\ref{subsec:BurgEntropy}  & $\left(\!\frac{X^{-1} + Y^{-1}}{2}\!\right)^{\!\!-1}$ & $\log\det\!\left(\!\frac{X+Y}{2}\!\right)-\frac{1}{2}\log\det(XY)$ & $\frac{1}{2}\log\det(XY) \;+\; \log\det\!\left(\frac{X^{-1}+Y^{-1}}{2}\right)$\\
\bottomrule 
\end{tabular}
}}
\end{table*}

\section{Examples}\label{sec:examples}
We next give some details for the three examples mentioned in Table \ref{table:BregmanExamples} since these are encountered frequently in applications. For the readers' convenience, a summary of the following is presented in Table \ref{tableFinalExamples}.

\subsection{The Case $\psi(X)=\|X\|_{\mathrm{F}}^2$}\label{subsec:SquaredFrob}
Here, the gradient and the Hessian are
$$\nabla \psi(X) = 2X\in\mathbb{S}^{n}_{++}\subset\mathbb{S}^{n},\; {\mathrm{H}}\psi(X) = 2\left(I_n\otimes I_n\right)\in\mathbb{S}^{n}_{++},$$
where $I_n$ denotes the $n\times n$ identity matrix, and $\otimes$ denotes the Kronecker product. 

In this case\footnote{We clarify here that $\psi(X)=\|X\|_{\mathrm{F}}^2$ is a Legendre-type function on the \emph{full} space $\mathbb{S}^n$ (a closed set with no boundary to blow up at), where it is strictly convex and $\nabla\psi$ is a global diffeomorphism onto $\mathbb{S}^n$. Example~III-A and Table~I row~1 should be understood for this function \emph{restricted} to $\mathbb{S}^{n}_{++}\subset\mathbb{S}^n$; all stated formulas remain valid there since $\mathbb{S}^{n}_{++}$ is an open subset of the ambient domain $\mathbb{S}^{n}$ on which restricted computation are performed.}, 
$$\psi^{*}(Y) = \frac{1}{4}{\mathrm{trace}}(Y^2), Y\in\mathbb{S}^{n}_{++},\quad\nabla \psi^{*}(Y) = \frac{1}{2}Y,$$
which from \eqref{ForwardMinimizer} and \eqref{BackwardMinimizer} yields  
$$\overrightarrow{M}_{\mathrm{canonical}}= \overleftarrow{M}_{\mathrm{canonical}} = \frac{X+Y}{2}.$$
Since the arithmetic mean is ${\mathrm{GL}}(n)$ invariant, the optimizer $\overleftarrow{M}_{\mathrm{canonical}}$ for this $\psi$ is, in fact, in $\mathcal{M}$.

Furthermore, $\psi^*(\nabla\psi(X)) = \frac{1}{4}\|2X\|_{\mathrm{F}}^2 = \|X\|_F^2$. From \eqref{JensenShannonSymmetrization} and \eqref{CanonicalReverseBregman}, the symmetrized Bregman divergences
$$D_{\psi}^{\overrightarrow{\mathrm{symm}}} = D_{\psi}^{\overleftarrow{\mathrm{symm}}} = \frac{1}{4}\|X-Y\|_{\mathrm{F}}^{2}.$$


\subsection{The Case $\psi(X)={\mathrm{trace}}(X\log X - X)$}\label{subsec:NegVonNeumannEntropy}
We find the gradient $\nabla \psi(X) = \log X\in\mathbb{S}^{n}$ where $\log$ denotes the principal logarithm, and the Hessian $${\mathrm{H}}\psi(X) = \int_{0}^{\infty}\left((tI_n + X)^{-1}\otimes (tI_n + X)^{-1}\right)\differential t\in\mathbb{S}^{n}_{++},$$
see e.g., \cite[Lemma 2(ii)]{halder2018gradient}. 

Direct computation gives $$\psi^{*}(Y) ={\mathrm{trace}}(\exp Y),\quad\nabla \psi^{*}(Y) = \exp Y.$$
From \eqref{ForwardMinimizer} and \eqref{BackwardMinimizer},  
$$\overrightarrow{M}_{\mathrm{canonical}}=\frac{X+Y}{2},\;\overleftarrow{M}_{\mathrm{canonical}} = \exp\left(\frac{\log X + \log Y}{2}\right),$$
namely the arithmetic and log-Euclidean means, respectively. Since the log-Euclidean mean is ${\mathrm{O}}(n)$ invariant, the optimizer $\overleftarrow{M}_{\mathrm{canonical}}$ is in $\widehat{\mathcal{M}}$ but not in $\mathcal{M}$.

From \eqref{JensenShannonSymmetrization}, the forward symmetrized Bregman divergence
\begin{align*}
D_{\psi}^{\overrightarrow{\mathrm{symm}}} &= \frac{1}{2}{\mathrm{trace}}(X\log X + Y\log Y)\\
&\qquad\qquad-{\mathrm{trace}}\left(\frac{X+Y}{2}\log\frac{X+Y}{2}\right).
\end{align*}
From \eqref{CanonicalReverseBregman} and that$$\psi^*(\nabla\psi(X)) = {\mathrm{trace}}(\exp(\log X)) = {\mathrm{trace}}(X),$$the reverse symmetrized Bregman divergence
\begin{align*}
D_{\psi}^{\overleftarrow{\mathrm{symm}}} = {\mathrm{trace}}\left(\dfrac{X+Y}{2}\right)-{\mathrm{trace}}\!\left(\!\exp\!\left(\frac{\log X+\log Y}{2}\right)\!\right).
\end{align*}


\subsection{The Case $\psi(X)=-\log\det X$}\label{subsec:BurgEntropy}
In this case, we compute the gradient 
$$\nabla \psi = -X^{-1}\in\mathbb{S}^{n}_{--}\subset\mathbb{S}^{n},$$
and the Hessian
$${\mathrm{H}}\psi(X) = \left(X^{-1}\otimes X^{-1}\right)\in\mathbb{S}^{n}_{++}.$$

We next find
$$\psi^{*}(Y) =-n-\log\det(-Y), Y\in\mathbb{S}^{n}_{--},\:\nabla \psi^{*}(Y) = -Y^{-1}.$$
Then, \eqref{ForwardMinimizer} and \eqref{BackwardMinimizer} give  
$$\overrightarrow{M}_{\mathrm{canonical}}=\frac{X+Y}{2},\;\overleftarrow{M}_{\mathrm{canonical}} = \left(\frac{X^{-1} + Y^{-1}}{2}\right)^{\!\!-1},$$
namely the arithmetic and harmonic means, respectively. Since the harmonic mean is ${\mathrm{GL}}(n)$ invariant, the optimizer $\overleftarrow{M}_{\mathrm{canonical}}$ for this $\psi$ is, in fact, in $\mathcal{M}$.

From \eqref{JensenShannonSymmetrization}, we obtain the forward symmetrized Bregman divergence
\begin{align*}
D_{\psi}^{\overrightarrow{\mathrm{symm}}} &= \log\det\left(\frac{X+Y}{2}\right)-\frac{1}{2}\log\det(XY),
\end{align*}
which is the $S$-divergence \cite{sra2016positive}. Using \eqref{CanonicalReverseBregman} and that $$\psi^*(\nabla\psi(X)) = -n - \log\det X^{-1} = \log\det X - n,$$ the reverse symmetrized Bregman divergence
\begin{align*}
D_{\psi}^{\overleftarrow{\mathrm{symm}}} =\frac{1}{2}\log\det(XY) \;+\; \log\det\!\left(\frac{X^{-1}+Y^{-1}}{2}\right).
\end{align*}


\section{Conclusion}\label{sec:conclusion}
Bregman divergences and their symmetrizations find broad usage across control, optimization and learning algorithms. For \emph{specific mirror maps} such as the Burg entropy and the negative von Neumann entropy, existing literature commonly use the arithmetic mean for symmetrizing the Bregman divergence on the cone positive definite matrices. Driven by the intuition that there might be variational principles lurking behind the scene, we investigated if there could be a notion of canonical mean for this symmetrization procedure for \emph{generic mirror maps}. We proposed variational formulations for the canonical means. We showed that for forward symmetrization, the arithmetic mean is canonical irrespective of the choice of the mirror map. For reverse symmetrization, we proved that the canonical mean is obtained by first computing the arithmetic mean in dual space, and then pulling it back to the primal space. We derived the corresponding symmetrized Bregman divergences in terms of their mirror maps. These results were exemplified via three mirror maps commonly used for the cone of positive definite matrices.

While we focused on the forward and reverse symmetrizations \eqref{defSymmetrizedBregmanDiv} and \eqref{defReverseSymmetrizedBregmanDiv}, as they are most common in literature, it is possible to extend our analyses for other symmetrizations, such as the geometric and harmonic a.k.a. resistor symmetrizations \cite{nielsen2019jensen}, and regularized variants \cite[Sec. 3]{acharyya2013bregman}. These will be pursued in future work.


\balance

\bibliographystyle{IEEEtran}
\bibliography{References.bib}

\end{document}